\documentclass[11pt]{amsart}
\usepackage[utf8]{inputenc}
\usepackage[english,frenchb]{babel}
\usepackage{amsmath}
\usepackage{amsfonts}
\usepackage{dsfont}
\usepackage{amssymb}
\usepackage{mathrsfs}
\usepackage{amsthm}
\usepackage{graphicx}
\usepackage[a4paper]{geometry}
\usepackage{enumitem}
\usepackage{hyperref}
\usepackage{cite}
\usepackage{url}

\DeclareMathOperator{\ch}{ch}
\DeclareMathOperator{\sh}{sh}
\DeclareMathOperator{\CF}{CF}

\DeclareMathOperator{\id}{id}

\DeclareMathOperator{\Cl}{Cl}

\DeclareMathOperator{\St}{St}
\theoremstyle{definition}
\newtheorem{definition}{Définition}
\newtheorem{notation}{Notation}
\theoremstyle{plain} 
\newtheorem{exemple}{Exemple}
\theoremstyle{plain} 
\newtheorem{remarque}{Remarque}

\newtheorem{theoreme}{Théorème}
\newtheorem{ThmIntro}{Théorème}

\newtheorem{ThmPrinc}{Théorème}

\newtheorem{lemme}[theoreme]{Lemme}
\newtheorem{cor}[theoreme]{Corollaire}
\newtheorem{prop}[theoreme]{Proposition}

\title{Polynomialité d'anneaux de représentations modulaires projectives}
\author{Hélène Pérennou}
\date{\today}
\thanks{L'auteur remercie le Centre Henri Lebesgue ANR-11-LABX-0020-01 ainsi que le projet ANR-16-CE40-0003 ChroK pour leur soutien.}
\address{Universit\'e de Nantes \\
Laboratoire de Math\'ematiques Jean Leray
(UMR 6629 CNRS \& UN)}
\email{helene.perennou@univ-nantes.fr}
\begin{document}
\selectlanguage{english}
\begin{abstract}
Consider the Grothendieck group of finite type projective modular representations of the symmetric groups on $n$ letters, or more generally, of its wreath product with a finite group. 
They form a graded group, with a product defined using induction. We show that the resulting graded ring is a polynomial ring.
\end{abstract}
\selectlanguage{frenchb}
\maketitle
\section*{Introduction}

On s'intéresse aux représentations modulaires des groupes symétriques $S_n$.
Pour $p$ premier, et pour un entier naturel $n$, notons $K_0^{\mathbb{F}_p}(S_n)$ le groupe de Grothendieck des $\mathbb{F}_pS_n$-modules projectifs de type fini.
Pour tous $(k,l)$ dans $\mathbb{N}^2$ tels que $k+l=n$, l'identification du groupe $S_k\times S_l$ au sous-groupe de $S_n$ qui stabilise le sous-ensemble $\{1,\ldots,k\}$ de $\{1,\ldots,n\}$ permet de définir le foncteur d'induction
$$ \_ \uparrow_{S_k\times S_l}^{S_n} : K_0^{\mathbb{F}_p}(S_k)\otimes K_0^{\mathbb{F}_p}(S_l) \rightarrow K_0^{\mathbb{F}_p}(S_n)$$
On note $K_0^{\mathbb{F}_p}(S_\infty)$ l'anneau gradué  dont le sous-anneau des éléments homogènes de degré $n$ est  $K_0^{\mathbb{F}_p}(S_n)$ et où la multiplication est donnée par ces foncteurs d'inductions.
On peut préciser la structure de cet anneau  :
\begin{ThmIntro}\label{thm1}
 L'anneau $K_0^{\mathbb{F}_p}(S_\infty)$  est un anneau de polynômes avec un générateur en chaque degré premier à $p$. 
\end{ThmIntro}
Ce résultat est sans doute connu des spécialistes, cependant nous n'avons pu trouver une référence. Il a été démontré par Robert Oliver et Pierre Vogel \cite{Vogel} en 1988, mais il semble que le tapuscrit ait été perdu.

On peut en fait être un peu plus général en considérant les produits en couronne d'un groupe fini $G$ quelconque avec les groupes symétriques. Soit $k$ un corps de caractéristique $p$. On note $K_0^{k}(G\wr S_n)$ le groupe de Grothendieck de la catégorie des $kG\wr S_n$-modules projectifs de type fini. Pour tous $(k,l)$ dans $\mathbb{N}^2$ tels que $k+l=n$, une identification naturelle du groupe $(G\wr S_k)\times (G\wr S_l)$ à un sous-groupe  de $G \wr S_n$ permet de définir le foncteur d'induction
$$ \_ \uparrow_{(G\wr S_k)\times (G\wr S_l)}^{G\wr S_n} : K_0^{k}(G\wr S_k)\otimes K_0^{k}(G\wr S_l) \rightarrow K_0^{k}(G\wr S_n)
.$$
On note alors $K_0^{k}(G\wr S_\infty)$ l'anneau gradué  dont le sous-anneau des éléments homogènes de degré $n$ est  $K_0^k(G\wr S_n)$ et où la multiplication est donnée par ces foncteurs d'induction. 
A nouveau, on peut en préciser la structure :
\begin{ThmIntro}\label{thm2}
Soit $k$ un corps de caractéristique $p$. Si $k$ est un corps de rupture pour le groupe fini $G$, l'anneau \hbox{$K_0^k(G\wr S_\infty)$} est un anneau de polynômes.
\end{ThmIntro}

Ce papier s'organise de la manière suivante. 
On commence par rappeler un résultat bien connu en caractéristique nulle (attribué à \cite{Geis} par \cite{Cartier}) : si on note $R^{\mathbb{Q}}(S_n)$ le groupe de Grothendieck des représentations de $S_n$ sur $\mathbb{Q}$ et $R^{\mathbb{Q}}(S_\infty)$ l'anneau gradué qui s'en déduit, il affirme que $R^{\mathbb{Q}}(S_\infty)$ est un anneau de polynômes sur $\mathbb{Z}$, avec pour générateurs les représentations triviales. 
Ensuite, la section \ref{wprod0:subsection} expose la généralisation du résultat précédent aux représentations des produits en couronne,
toujours en caractéristique nulle \cite[7]{Zel}:
si $K$ est un corps de caractéristique nulle, assez grand, $R^K(G\wr S_\infty)$ est isomorphe à un produit tensoriel fini de copie de $R^{K}(S_\infty)$. C'est donc également un anneau de polynôme sur $\mathbb{Z}$.
On se réfère essentiellement à \cite{Zel} pour ces résultats en caractéristique nulle, on se place donc dans le cadre des $PSH$-algèbres. 

La seconde partie démontre le théorème \ref{thm1}. Celle-ci n'utilise d'outil supplémentaire que la théorie de Brauer qui établit un lien entre les représentations modulaires et les représentations en caractéristique nulle.
Ce sont des arguments similaires qui démontrent le théorème plus général \ref{thm2} dans la troisième partie.

\section{Anneau de représentations en caractéristique nulle}
Dans cette section on rappelle les résultats de \cite{Zel} que nous utiliserons par la suite. Pour les détails, on pourra aussi consulter \cite[3]{GinbergReiner}. 
\subsection{Cas des groupes symétriques}\label{cpxsym}

\begin{notation}\label{notsym}
\item
\begin{itemize}[label=--,leftmargin=*]
\item $K$ désigne un corps de caractéristique $0$.
\item Pour $n\in\mathbb{N}$, $\mathcal{P}_n$ désigne l'ensemble des partitions de $n$ et \hbox{$\mathcal{P} = \bigsqcup_{n\geq 0} \mathcal{P}_n$}  l'ensemble des partitions. 
On munit $\mathcal{P}$ de l'ordre lexicographique. 
Le plus grand élément de $\mathcal{P}_n$ est la partition avec une unique part de taille $n$.
\item Pour $\mathbf{t}=(t_1,t_2,\ldots)$ une suite d'indéterminées, on note $\Lambda(\mathbf{t})=\Lambda$ l'anneau gradué des fonctions symétriques en $\mathbf{t}$ à coefficients dans $\mathbb{Z}$. On note, pour un entier positif $n$:
$$ h_n = \sum_{i_1 \leq \ldots \leq i_n} t_{i_1}\ldots t_{i_n}$$
avec la convention $h_0=1$.
\end{itemize}

Rappelons un résultat classique sur les séries symétriques.
\begin{theoreme}
L'anneau des séries formelles est un anneau de polynômes. 
Précisément :
$$\Lambda = \mathbb{Z}[h_i | i\geq 1 ]$$
\end{theoreme}
En fait, c'est une $PSH$-algèbre pour la multiplication et la co-multiplication usuelles \cite[5]{Zel}.
\begin{itemize}[label=--,leftmargin=*]
\item On note $R^K(S_n)$ le groupe de Grothendieck de la catégorie des représentations du groupe symétrique $S_n$ sur $K$. On note également $R^K(S_\infty)$ le groupe abélien gradué dont les éléments homogènes de degré $n$ forment $R^K(S_n)$. C'est une algèbre de Hopf, la multiplication étant induite par les foncteurs d'induction $(\_)\uparrow_{S_n\times S_m}^{S_{n+m}}$ et la co-multiplication  par les foncteurs de restriction \hbox{$(\_)\downarrow_{S_n\times S_m}^{S_{n+m}}$} (voir \cite{Zel} pour les détails, notamment du fait que la co-multiplication est un morphisme d'anneaux). On définit un produit scalaire sur $R^K(S_\infty)$ en posant 
$$\langle \chi_\lambda, \chi_{\lambda'} \rangle =\delta_{\lambda,\lambda'}$$
sur la base $\{\chi_\lambda\}_{\lambda\in\mathcal{P}}$ des simples $\{\chi_\lambda\}_{\lambda\in\mathcal{P}}$.

\item On note $\CF^K(S_\infty)$ l'anneau gradué des fonctions de classes des $S_n$. Prenant les caractères, on plonge $R^K(S_\infty)$ dans $\CF^K(S_\infty)$, et la structure s'étend pour faire de $\CF^K(S_\infty)$ une algèbre de Hopf, qu'on identifie à $K\otimes R^K(S_\infty)$. 
\item On note $c_n=n\mathds{1}_n$, où $\mathds{1}_n$ est la classe caractéristique du cycle d'ordre $n$ dans $S_n$. Si $\lambda = (\lambda_1,\ldots,\lambda_k) \in \mathcal{P}_n$, on note $c_{\lambda}=c_{\lambda_1}\ldots c_{\lambda_k}$. 

On a \cite[6.3]{Zel}:
$$ \CF^K(S_\infty) = K[ c_n |  n\geq 1]$$
\end{itemize}
\end{notation}
Le résultat de structure des représentations des groupes symétriques en caractéristique nulle est le suivant :
\begin{theoreme}\cite[6]{Zel}
On dispose d'un isomorphisme de $PSH$-algèbres
$$ R^K(S_\infty) \overset{\sim}{\longrightarrow} \Lambda $$
qui envoie la base orthonormée des caractères irréductibles sur la base orthonormée des fonctions de Schur.
De plus cet isomorphisme envoie $x_n$ sur $h_n$ pour tout entier positif $n$.
En particulier $$R^K(S_\infty)=\mathbb{Z}[x_n | n\geq 1]$$
\end{theoreme}

\subsection{Cas des produits en couronne}
\label{wprod0:subsection}
On donne ici l'analogue du résultat précédent pour les représentations des produits en couronne avec les groupes symétriques.

\begin{notation}
	\item
\begin{itemize}[label=--,leftmargin=*]
		\item $G$ désigne un groupe fini et pour tout $n$ dans $\mathbb{N}$, on note $G\wr S_n$ le produit en couronne de $G$ avec $S_n$.
		\item A présent, $K$ est un corps de rupture pour $G$, toujours de caractéristique $0$.
		\item On note $\Cl(G)$ l'ensemble des classes de conjugaison de $G$ et $N=\# \Cl(G)$.
		\item De même, on note $\Cl(G\wr S_n)$ l'ensemble des classes de conjugaison de $G\wr S_n$. Ces classes sont paramétrées par l'ensemble  $\{\psi : \Cl(G) \rightarrow \mathcal{P} | \sum_{C\in\Cl(G)}|\psi(C)| = n\}$.
		\item  On note $\Omega(G)$ l'ensemble  des classes d'équivalence des représentations irréductibles de $G$ sur $K$ et $\{\rho_1,\ldots,\rho_N\}$ ces représentations irréductibles.
		\item $R^K(G \wr S_\infty)$ est l'anneau gradué dont la multiplication est donnée par les foncteurs d'induction et la co-multiplication donnée par les foncteurs de restriction. C'est une algèbre de Hopf \cite[7.1]{Zel}.
		\item $\CF^K(G\wr S_\infty)$ l'anneau gradué des fonctions de classe des $G\wr S_n$. A nouveau la multiplication se déduit de celle sur $R^K(G\wr S_\infty)$. 
		\item Soit $C=[g]$ la classe de conjugaison associée à $g$ dans $G$, on note $$\xi_{n,c} = \frac{|G\wr S_n|}{|C|}\mathds{1}_{n,C}$$ où $\mathds{1}_{n,C}$ est la fonction caractéristique de la classe de $g.(1 \ldots n)$ dans $G\wr S_n$. L'anneau $\CF^K(G\wr S_\infty)$ est une algèbre de Hopf qui s'identifie à $K\otimes R^K(G \wr S_\infty)$ et on a \cite[7.7]{Zel}:
		$$\CF^K(G\wr S_\infty)=K[\xi_{n,C} | C\in\Cl(G), n\geq 1]$$
\end{itemize}
\end{notation}

\begin{definition}\label{defwrp}
Soit $\chi_\lambda$ une représentation irréductible de $S_n$ et $\rho\in\Omega(G)$. On définit $\chi_{\lambda,\rho}$ une représentation de $G\wr S_n$ qui agit sur l'espace de la représentation $\chi_{\lambda}\otimes(\rho^{\otimes n})$ par :
$$\sigma(u\otimes(v_1\otimes \ldots \otimes v_n))=\sigma(u)(v_{\sigma^{-1}(1)}\otimes\ldots\otimes v_{\sigma^{-1}(n)}$$
$$\mathbf{g}(u\otimes(v_1\otimes \ldots \otimes v_n))= u\otimes(g_1v_1\otimes \ldots \otimes g_nv_n))$$
où $\sigma\in S_n$ et $\mathbf{g}=(g_1,\ldots,g_n)\in G^n$.

Ces constructions induisent pour chaque $\rho\in\Omega(G)$ un morphisme $\Phi_\rho$ définit sur la base des classes d'irréductibles par : 
$$ 
\begin{array}{rcl}
\Phi_\rho : R^K(S_\infty) & \rightarrow & R^K(G\wr S_\infty) \\
\chi_\lambda   & \mapsto & \chi_{\lambda,\rho}
\end{array}$$
On notera plus simplement $\Phi_k$ pour $\Phi_{\rho_k}$.
\end{definition}

\begin{definition}
	Soit $\varphi : \Omega(G)\rightarrow \mathcal{P}$ est tel que $\sum_{\rho\in\Omega(G)} |\varphi(\rho)|=n$. On a 
	$$ (\chi_{\varphi(\rho_1),\rho_1} \otimes\ldots\otimes \chi_{\varphi(\rho_N),\rho_N})$$
	une représentation de $(G\wr S_{|\varphi(\rho_1)|})\times \ldots\times (G\wr S_{|\varphi(\rho_N)|})$. Ce groupe s'identifie à un sous-groupe de $G\wr S_n$ et on note $\chi_\varphi$ la représentation induite sur $G\wr S_n$.
\end{definition}

\begin{theoreme}
Les représentations irréductibles de $G\wr S_n$ sur $K$ sont les représentations $\chi_{\varphi}$ où $\varphi$ parcours l'ensemble des applications  $\Omega(G)\rightarrow \mathcal{P}$ est tel que $\sum_{\rho\in\Omega(G)} |\varphi(\rho)|=n$. De plus, on a un isomorphisme de $PSH$-algèbres :
$$
\begin{array}{rcl}
 R^K(G\wr S_\infty) & \rightarrow  & \bigotimes_{i=1}^N R^K(S_\infty) \\
        \chi_{\varphi} & \mapsto & \chi_{\varphi(\rho_1),\rho_1} \otimes\ldots\otimes \chi_{\varphi(\rho_N),\rho_N}
\end{array}$$
et l'isomorphisme réciproque est $\Phi=\Phi_{1}\otimes\ldots \otimes \Phi_{N}$.
En particulier,
$$R^K(G\wr S_\infty) = \mathbb{Z}[\Phi_k(x_n) | k \in [\![1;N]\!], n\geq 1] $$
\end{theoreme}

On dispose donc pour chaque $n$ de deux bases de l'espace des éléments primitifs de $\CF^K(S_n)$, à savoir $\{\Phi_\rho(c_n) | \rho\in\Omega(G) \}$ et $\{ \xi_{n,C} | C\in \Cl(G)\}$. La proposition suivante donne les formules changement de bases.

\begin{prop}\cite[p. 103]{Zel}\label{chgtbase}
	Pour tout $\rho\in\Omega(G)$ et $C\in\Cl(G)$, on note $\rho(C)$ la valeur du caractère de $\rho$ sur la classe de conjugaison $C$. Alors, pour tout $n\geq 1$, on a :
	$$\xi_{n,C} = \sum_{\rho\in\Omega(G)} \overline{\rho(C)}.\Phi_{\rho}(c_n) $$
	et
	$$\Phi_{\rho}(c_n) = \sum_{C\in \Cl(G)} \rho(C).\frac{|C|}{|G|}.\xi_{n,C}$$
\end{prop}

\section{Anneau des représentations modulaires projectives des groupes symétriques}

On peut à présent s'intéresser aux représentations modulaires projectives des groupes symétriques. On prouve le théorème \ref{thm1} à la fin de cette section.

\begin{notation}
\item
\begin{itemize}[label=--,leftmargin=*]
\item Dans toute la suite, $p$ désigne un nombre premier fixé ;
\item On se donne $(K,A,k)$ un système $p$-modulaire, i.e. $A$ est un anneau de valuation d'idéal maximal $\mathfrak{m}$, $k=A/\mathfrak{m}$ et $K$ est le corps des fractions de $A$ ;
\item On note $\mathcal{P}_{n,reg}$ l'ensemble des partitions où toutes les parts sont de longueur première à $p$ et $\mathcal{P}_{reg}=\bigsqcup_{n \geq 0} \mathcal{P}_{n,reg}$.
\end{itemize}

\begin{remarque}
Les classes de conjugaison $p$-régulières de $S_n$ (i.e. d'élément d'ordre premier à $p$) sont paramétrées par $\mathcal{P}_{n,reg}$.
\end{remarque}

\begin{itemize}[label=--,leftmargin=*]
\item On note $K_0^k(S_n)$ (resp. $K_0^A(S_n)$) le groupe de Grothendieck de la catégorie des $kS_n$-modules (resp. $AS_n$-modules) projectifs de type fini et $K_0^k(S_\infty)$ (resp. $K_0^A(S_\infty)$) le groupe gradué dont les éléments homogènes de degré $n$ s'identifient aux éléments de $K_0^k(S_n)$ (resp. $K_0^A(S_n)$).
\end{itemize}

\begin{remarque}
Comme tout corps est un corps de rupture pour les $(S_n)_{n\geq 0}$ \cite[2.1.12]{JamesKerber}, on a toujours $R^K(S_\infty) \cong R^{\mathbb{Q}}(S_\infty)$ et $K_0^k(S_\infty) \cong K_0^{\mathbb{F}_p}(S_\infty)$.
\end{remarque}

\begin{itemize}[label=--,leftmargin=*]
\item On note $\CF^K_{reg}(S_\infty)$ l'ensemble des fonctions de classes à valeurs dans $K$ qui s'annulent sur les éléments d'ordre divisible par $p$. L'anneau $K\otimes K_0^k(S_\infty)$ s'identifie à  $\CF^K_{reg}(S_\infty)$ \cite[18.3]{Serre} et $\CF^K_{reg}(S_\infty) = K[ c_n |  p\nmid n]$.
\end{itemize}

Avant de donner les dernières notations, on rappelle quelques éléments de théorie des représentations modulaires. Soit $G$ un groupe et $(K,A,k)$ un système $p$-modulaire, avec $k=A/\mathfrak{m}$.

\begin{prop}\cite[14.4]{Serre}
 La réduction modulo $\mathfrak{m}$ induit un isomorphisme entre $K^A_0(G)$ et $ K^k_0(G)$.
\end{prop}

\begin{definition}\label{defe}
En composant l'isomorphisme réciproque de la réduction mod. $\mathfrak{m}$ avec le morphisme induit par l'extension des scalaires de $A$ vers $K$, on définit un morphisme $e_G : K^k_0(G)\rightarrow R^K(G)$ 
\end{definition}

\begin{prop}\cite[16.1,2]{Serre}\label{inje}
Le morphisme $e_G$ est injectif.
\end{prop}
On peut caractériser l'image de $e_G$.
\begin{prop}\cite[16.2]{Serre}\label{0sing}
L'image du morphisme $K^k_0(G)\rightarrow R^K(G)$ est l'ensemble des éléments de $R(G)$ dont le caractère est nul sur les éléments d'ordre divisible par $p$.
\end{prop}

\begin{itemize}[label=--,leftmargin=*]
\item On note $\varepsilon_n=e_{S_n} : K^k_0(S_n) \rightarrow R^K(S_n)$ le morphisme de la définition \ref{defe} pour $G=S_n$. On note $\varepsilon : K^k_0(S_\infty) \rightarrow R^K(S_\infty)$ le morphisme induit par la famille $\{\varepsilon_n\}_{n\in\mathbb{N}}$.
\end{itemize}
\end{notation}

\begin{lemme}\label{ime}
L'application $\varepsilon$ induit un isomorphisme de $K_0^k(S_\infty)$ vers le sous-anneau de $R^K(S_\infty)$ formés des éléments dont le caractère est dans $\CF^K_{reg}(S_\infty)$.
\end{lemme}
\begin{proof}
D'après les propositions \ref{inje} et \ref{ime}, les morphismes $\varepsilon_n$ sont injectifs et ont pour images les classes de représentations dont le caractère est nul sur les éléments d'ordre divisible par $p$. Il reste donc à vérifier que $\varepsilon$ est multiplicatif. D'après la définition de $\varepsilon$ (voir la définition \ref{defe}), il suffit de remarquer l'induction commute avec l'extension des scalaires.
\end{proof}

Avant d'énoncer le théorème principal pour les groupes symétriques, on donne un résultat sur les séries formelles :

\begin{definition}
Soit $X(t)=\sum_{n \geq 0} x_n t^n$ une série formelle et $p$ un nombre premier.
\begin{itemize}[label=--]
\item On appelle partie \emph{$p$-régulière} de $X$ la série formelle : $\sum_{p\nmid n}  x_n t^n$ ;
\item On appelle partie \emph{$p$-singulière} de $X$ la série formelle : $\sum_{n \geq 0}  x_{np} t^{np}$.
\end{itemize}
\end{definition}

\begin{definition}
Soit $p$ un nombre premier.
On note $\sh_p$ la partie $p$-régulière de $\exp$ et $\ch_p$ sa partie $p$-singulière.
\end{definition}

\begin{lemme}\label{lemmetech}
	Soit $p$ un nombre premier et soit $X(t)= \sum_{n\geq 0} x_n t^n$ et \hbox{$C(t) = \sum_{n\geq 1} \frac{c_n}{n} t^n$} deux séries formelles telles que 
	$$ X(t) = \exp(C(t))$$
	On note $U(t)$ (resp. $V(t)$) la partie $p$-singulière (resp. $p$-régulière) de $X(t)$, $A(t)$ (resp. $B(t)$) la partie $p$-singulière (resp. $p$-régulière) de $C(t)$, et on note $Y(t) = \sum_{n\geq 0} y_n t^n$ la série formelle vérifiant :
	$$Y(t)= \frac{V(t)}{U(t)}$$
	Alors,
	\begin{enumerate}
		\item $y_n=0$ pour $p | n$ ;
		\item $y_n$ est un polynôme en les $c_i$, $p \nmid i$ ;
		\item $(y_n - x_n)$ est combinaison linéaire de $x_\lambda$ avec $\lambda \notin \mathcal{P}_{n,reg}$.		
	\end{enumerate}
\end{lemme}
\begin{proof}
	On a 
	\begin{eqnarray*}
		\exp(C(t)) & = & \exp(A(t)+B(t)) \\
		& = & \exp(A(t)).\ch_p(B(t))+\exp(A(t)).\sh_p(B(t))
	\end{eqnarray*}	
	En identifiant les parties singulières et régulières, on obtient :
	$$U(t) = \exp(A(t)).\ch_p(B(t))$$ 
	et 
	$$ V(t)=\exp(A(t)).\sh_p(B(t))$$
	Donc 
	$$Y(t) = \frac{\sh_p(B(t))}{\ch_p(B(t))}$$
	et les $y_n$, $n\geq 0$, vérifient les deux premiers points.	
	
	Par ailleurs, on a :
		
	\begin{eqnarray*}
		\frac{1}{U(t)} & = & \frac{1}{1 + \sum_{n\geq 1} x_{pn}t^{pn}} \\
		& = &\sum_{k\geq 0} (-1)^k (\sum_{n\geq 1} x_{pn}t^{pn})^k \\
		& = & \sum_{\substack{k\geq 0,\\ p \mid \lambda_1,\ldots, \lambda_k,\lambda_i>0 }}(-1)^{\lambda_1+\ldots \lambda_k}x_{\lambda_1,\ldots,\lambda_k} t^{\lambda_1 + \ldots +\lambda_k}
	\end{eqnarray*}
	Donc :
	\begin{eqnarray*}
		Y(t) & = & V(t)\big[\sum_{\substack{k\geq 0,\\p\mid \lambda_1,\ldots, \lambda_k,\lambda_i>0 }} (-1)^{\lambda_1+\ldots \lambda_k}x_{\lambda_1,\ldots,\lambda_k} t^{\lambda_1 + \ldots +\lambda_k}\big] \\
	\end{eqnarray*}
	Ce qui prouve le dernier point.
\end{proof}

\begin{remarque}\label{remyi}
La preuve fournit une formule explicite pour les $y_n$ : 
\begin{eqnarray*}y_n=\sum_{\substack{(\lambda_0,\ldots,\lambda_k)\in (\mathbb{N}^*)^{k+1} \\ \sum_i\lambda_i = n \\ p\nmid \lambda_0 \text{ et } p\mid \lambda_i, i\geq 1 }} (-1)^k x_{\lambda_0}\ldots x_{\lambda_k} \end{eqnarray*}
On a par exemple, pour $p=2$ :
\begin{itemize}[font=\footnotesize]
\item $y_1=x_1$;
\item $y_3 = x_3 - x_1x_2$ ;
\item $y_5 = x_5-x_3x_2+x_1x_2^2-x_1x_4 $ ;
\item $y_7 = x_7-x_5x_2 - x_3x_4 + x_3x_2^2 - x_1x_6+2x_1x_4x_2-x_1x_2^3$.
\end{itemize}
\end{remarque}

On peut à présent énoncer le résultat principal de cette section :

\begin{ThmPrinc}
 L'anneau $K_{0}^k(S_\infty)$ est un anneau de polynômes avec un générateur en chaque degré premier à $p$.
Précisément :
$$\varepsilon (K_0^k(S_\infty))=\mathbb{Z}[y_i | p\nmid i]$$ 
\end{ThmPrinc}
\begin{proof}
D'après le lemme \ref{ime}, il suffit de montrer que $\varepsilon(K_{0}^k(S_\infty))=\mathbb{Z}[y_i | p\nmid i]$. En reprenant les notations de \ref{cpxsym}, on a l'égalité dans $\CF^K(S_\infty)$ :
$$ x_n = \sum_{(1^{n_1},\ldots,k^{n_k})\in\mathcal{P}_n} \bigg(\prod_{i=1}^k \frac{ c_i^{n_i}}{ n_i!i^{n_i}}\bigg)$$
Ainsi en calculant dans $\CF^K(S_\infty)[[t]]$, on obtient :
$$ 
\sum_{n\geq 0} x_n t^n= \exp\big(\sum_{i\geq 1} \frac{c_i}{i} t^i\big)
$$
Donc d'après le lemme \ref{lemmetech} et en reprenant les notations, on a pour tout $n$ premier à $p$, $y_n$ dans $\varepsilon(K_{0}^k(S_\infty))$ et donc $\mathbb{Z}[y_n | p\nmid n]$ est un sous-anneau de $\varepsilon(K_{0}^k(S_\infty))$. De plus, en identifiant les représentations avec leur caractère et en comparant les dimensions en chaque degré, on a également :
\begin{equation}\label{polyK}
K[y_n | p\nmid n] = \CF^K_{reg}(S_\infty)
\end{equation} 

Soit maintenant $z\in \varepsilon(K_{0}^k(S_\infty))$ homogène de degré $n$. En plongeant $z$ dans $K\otimes \varepsilon(K_{0}^k(S_\infty))$ par le morphisme canonique et d'après (\ref{polyK}) on obtient :
$$ z = \sum_{\lambda\in\mathcal{P}_{n,reg}}\alpha_\lambda y_\lambda $$
avec $\alpha_\lambda\in K$.
De plus, $z$ est également un élément de $R^K(S_\infty)$, donc 
$$ z = \sum_{\lambda\in\mathcal{P}_n}\beta_\lambda x_\lambda $$
avec $ \beta_\lambda \in \mathbb{Z}$.
On a donc :
$$
0 =  \sum_{\lambda\in\mathcal{P}_n}\beta_\lambda x_\lambda  - \sum_{\lambda\in\mathcal{P}_{n,reg}}\alpha_\lambda y_\lambda
$$
D'après la définition des $y_n$,
$$y_\lambda = x_\lambda + R_\lambda$$
avec $R_\lambda$ combinaison linéaire des $\{x_{\lambda'}\}_{\lambda'\notin \mathcal{P}_{n,reg}}$.
Donc,
$$
0  =  \sum_{\lambda\in\mathcal{P}_{n,reg}} (\beta_\lambda - \alpha_\lambda)x_\lambda + \sum_{\lambda\notin \mathcal{P}_{n,reg}}\beta_\lambda x_\lambda - \sum_{\lambda\in \mathcal{P}_{n,reg}}\alpha_\lambda R_\lambda
$$
Enfin, la famille $\{x_\lambda\}_{\lambda\in\mathcal{P}_n}$ est une base de $R^K(S_\infty)_n=R^K(S_n)$. On en déduit que 
$$\sum_{\lambda\in\mathcal{P}_{n,reg}} (\beta_\lambda - \alpha_\lambda)x_\lambda  = 0$$
puis que pour tout $\lambda\in\mathcal{P}_{n,reg}$, $$\beta_\lambda - \alpha_\lambda = 0$$

Finalement $z\in\mathbb{Z}[y_n | p\nmid n]$ et le résultat est prouvé.
\end{proof}

\begin{exemple}
Comparons les premiers générateurs du théorème et les $\mathbb{F}_2S_n$-modules projectifs indécomposables  : 
\begin{itemize}[font=\footnotesize]
\item $y_1=x_1$ est la représentation triviale du groupe trivial. Notons au passage que pour tout $n\geq 1$, $x_1^n$ est la classe de la représentation régulière ;
\item Il y a une unique représentation simple de $S_2$ sur $\mathbb{F}_2$ : la représentation triviale. Donc un unique  $\mathbb{F}_2S_2$ module projectif indécomposable : la couverture projective de la triviale qui est $\mathbb{F}_2S_2$ dont la série de composition est deux fois la triviale et $2x_2 = y_1^2$ ;
\item Pour $n=3$, on a la décomposition $\mathbb{F}_2S_3 = P_{\mathbb{F}_2} \oplus \St_2^{\oplus 2}$ où $P_{\mathbb{F}_2}$ est la couverture projective de la représentation triviale de $S_3$ sur $\mathbb{F}_2$ et $\St_2$ est la seconde représentation simple de $S_3$. D'après ce qui précède, on a $x_1^3 = 2x_1x_2$, donc $x_1^3 = 2x_3 + 2(x_1x_2-x_3)$. En particulier, $(-y_3)= x_1x_2-x_3$ est la classe de $\St_2$ ;
\item Il y a deux représentations simples de $S_4$ sur $\mathbb{F}_2$ \cite[p.267]{Webb}. Notons $P_1$ la couverture projective de la triviale et $P_2$ la couverture projective du second simple. On vérifie que $\mathbb{F}_2S_4= P_1\oplus P_2^{\oplus 2}$. Le calcul (en utilisant par exemple \cite{reps}) permet d'identifier $(-y_1y_3)$ à la classe de $P_2$. Il s'en suit que la couverture projective de la triviale appartient à la classe $y_1^4 +2y_1y_3$.
\end{itemize}
Notons que la remarque \ref{remyi} donne une formule explicite des $4$ premiers générateurs en fonction des classes des représentations triviales.
\end{exemple}

\section{Anneau des représentations modulaires projectives des produits en couronne}

Le cas des produits en couronne avec les groupes symétriques demande un peu plus de travail mais la preuve du théorème \ref{thm2} est très similaire à la précédente.

\begin{notation}
	\item
	\begin{itemize}[label=--,leftmargin=*]
	 	\item $p$ est un nombre premier fixé.
		\item $G$ est un groupe fini.
		\item $G_{reg}$ désigne l'ensemble des éléments de $G$ d'ordre premier à $p$.
		\item $\Cl_{reg}(G)$ l'ensemble des classes de conjugaison $p$-régulières de $G$, i.e. des classes d'éléments de $G_{reg}$.
		\end{itemize}		
\begin{remarque}
Les classes de conjugaison $p$-régulières de $G\wr S_n$ sont paramétrées par l'ensemble 
$$\{\varphi : \Cl_{reg}(G)\rightarrow \mathcal{P}_{reg} | \sum_{C\in\Cl_{reg}(G)} |\varphi(\rho)|=n\} $$
\end{remarque}
\begin{itemize}[label=--,leftmargin=*]
		\item $(K,A,k)$ un système $p$-modulaire de rupture pour $G$. D'après \cite[4.4.8]{JamesKerber} c'est également un système $p$-modulaire de rupture pour les $G\wr S_n$.
		\item On note $K_0^k(G\wr S_n)$ le groupe de Grothendieck de la catégorie des $kG\wr S_n$-modules projectifs de type fini et $K_0^k(G\wr S_\infty)$ l'algèbre graduée qui se déduit des $K_0^k(G\wr S_n)$ pour $n$ dans $\mathbb{N}$.
		\item On note $R^K(G)$ le groupe de Grothendieck  de la catégorie des représentations de $G$ sur $K$ et $K^k_0(G)$ le groupe de Grothendieck de la catégorie des $kG$-modules projectifs de type fini.
		\item On note plus simplement $e = e_G : K^k_0(G) \rightarrow R^K(G)$ le morphisme de la définition \ref{defe}. 
		\item On désigne par $\varepsilon : R^K(G\wr S_\infty) \rightarrow K_0^k(G\wr S_\infty)$ le morphisme d'anneaux gradués induit par les $(e_{G\wr S_n})_{n\in\mathbb{N}}$ (voir à nouveau la définition \ref{defe}).
		\item On rappelle que $N=\dim_K \CF^K(G)$.
		\item On note $\CF^K_{reg}(G)$ le sous-espace de $\CF^K(G)$ formé des fonctions qui s'annulent en dehors de $G_{reg}$. D'après le lemme \ref{ime}, ce sous-espace s'identifie à $K\otimes e(K_0^k(G))$. On note $M= \dim_K \CF_{reg}^K(G)$.
		\item On note $\chi_1,\ldots, \chi_N$ les caractères irréductibles de $G$ sur $K$. Ils forment une base de $\CF^K(G)$.
		\item On note $\varphi_1,\ldots,\varphi_M$ l'image par $K\otimes e$ des caractères de Brauer des $kG$-modules projectifs indécomposables. Ils forment une base de $\CF_{reg}(G)$.
		\item On note $E\in \mathcal{M}_{N,M}(\mathbb{Z})$ la matrice de l'application $e : K_0^k(G) \rightarrow R^K(G)$ dans les bases données ci-dessus.
	\end{itemize}
\end{notation}

\begin{lemme}
	Le morphisme $e$ admet une rétraction, i.e il existe une application $\mathbb{Z}$-linéaire $\sigma :R^K(G) \rightarrow K_0^k(G)$ telle $\theta\circ e = \id_{K_0^k(G)}$. 
\end{lemme}
\begin{proof}
	D'après \cite{Serre}, $E=D^t$, où $D$ est la matrice de l'application $\mathbb{Z}$-linéaire $d$. De plus, $d$ admet une section $\sigma$ \cite[18.4]{Serre}. On en déduit que $e$ admet une rétraction.
\end{proof}

\begin{itemize}[label=--,leftmargin=*]
\item Puisque $e$ admet une rétraction, on peut compléter $\{\varphi_1,\ldots,\varphi_M\}$ en une base $\{\varphi_1,\ldots,\varphi_N\}$ de $R^K(G)$. On note pour tout $i\in [\![1,N]\!]$ :
$$ \varphi_i =
\begin{pmatrix}
\varphi_{1,i} \\
\vdots \\
\varphi_{N,i}
\end{pmatrix}$$
\end{itemize}

\begin{definition}
	Pour tout $k\in \{ 1,\ldots,M\}$,
	$$ X_k(t) = \sum_{i\geq 0} X_{k,i} t^i = \big( \sum_{i\geq 0} \Phi_1(x_i)t^i\big)^{\varphi_{1,k}} \ldots \big( \sum_{i\geq 0} \Phi_M(x_i)t^i\big)^{\varphi_{M,k}} $$
	et pour tout $k\in \{ M+1,\ldots,N\}$,
	$$ X_k(t)= \sum_{i\geq 0} X_{k,i} t^i =  \sum_{i\geq 0}\sum_{j=1}^N \varphi_{j,k}\Phi_j(x_i) t^i $$
	où les $\Phi_1,\ldots, \Phi_N$ sont les applications de la définition \ref{defwrp}.
\end{definition}

\begin{lemme}\label{lt2}
	Pour tout $k\in \{ 1,\ldots,M\}$, on a l'égalité suivante dans $\CF^K(G\wr S_\infty)$ :
	$$ X_k(t) = \exp(C_k(t)) $$
	avec $$C_k(t) = \sum_{i=1}\big[\sum_{C\in \Cl_{\text{reg}}(G)} \big( \sum_{j=1}^N \frac{|C|}{|G|} \varphi_{j,k} \chi_j(C)\big) \frac{\xi_{i,C}}{i} t^i \big]$$
\end{lemme}
\begin{proof}
	\begin{eqnarray*}
		X_k(t)  & = & \big( \sum_{i\geq 0} \Phi_1(x_i)t^i\big)^{\varphi_{1,k}} \ldots \big( \sum_{i\geq 0} \Phi_M(x_i)t^i\big)^{\varphi_{M,k}}\\
		& = & \exp\big( \sum_{i\geq 1} \Phi_1(\frac{c_i}{i})t^i\big)^{\varphi_{1,k}}\ldots \exp\big( \sum_{i\geq 1} \Phi_M(\frac{c_i}{i})t^i\big)^{\varphi_{M,k}}\\
		& = & \exp\big(\sum_{i\geq 1} ( \sum_{j=1}^N \varphi_{j,k} \Phi_j(\frac{c_i}{i} t^i)\big) \\
		& = & \exp\big(\sum_{i\geq 1} \big[ \sum_{j=1}^N \big(\sum_{C\in \Cl(G)}\varphi_{j,k}\chi_k(C)\frac{|C|}{|G|}\frac{\xi_{i,C}}{i} t^i\big)\big]\big)\big)
	\end{eqnarray*}
	car $\Phi_j(c_i) = \sum_{C\in\Cl(G)} \chi_j(C)\frac{|C|}{|G|} \xi_{i,C}$
	\begin{eqnarray*}
		&= & \exp\big[\sum_{C\in\Cl(G)}\big(\sum_{j=1}^N \varphi_{j,k} \xi_j(C) \big) \frac{\xi_{i,C}}{i} t^i \big]\\
		&= & \exp\big[\sum_{C\in\Cl_{\text{reg}}(G)}\big(\sum_{j=1}^N \varphi_{j,k} \xi_j(C) \big) \frac{\xi_{i,C}}{i} t^i \big]
	\end{eqnarray*}
	D'après le lemme \ref{0sing}.
\end{proof}

\begin{cor}\label{cor}
	Pour tout $k\in\{1,\ldots,M\}$, on note :
	$$Y_k(t)= \sum_{n\geq 0} y_{k,n} t^n$$
	défini comme dans le lemme \ref{lemmetech}. On observe alors que pour tout $n\geq 1$, $y_{k,n}$ est combinaison linéaire de $\xi_C$ avec $C\in\Cl_{\text{reg}}(G\wr S_n)$. En particulier, $y_{k,n}\in  K\otimes \varepsilon(K_0^k(G\wr S_\infty))$, pour tout $n\geq 1$ et $k\in[\![1,M]\!]$.
\end{cor}
\begin{proof}
	 En utilisant les lemmes \ref{lemmetech} et \ref{lt2}, on vérifie que pour tout $n\geq 1$ et $k$ dans $[\![1,M]\!]$, $y_{i,n}$ est combinaison linéaire à coefficients dans $K$ de $\xi_C$ avec \hbox{$C\in \Cl_{\text{reg}}(G\wr S_n)$.}
		 Par ailleurs, l'image de $K\otimes \varepsilon$ dans $\CF^K(G\wr S_\infty)$ est exactement l'ensemble des fonctions de classe qui s'annulent sur les classes $p$-singulières. D'après la première partie de la preuve, c'est le cas des $y_{i,n}$, pour tout $n\geq 1$ et $k\in[\![1,M]\!]$.
\end{proof}

\begin{lemme}\label{chtgen}
	L'ensemble $\{X_{k,i} | k\in[\![1;N]\!], i\geq 1\}$ est un ensemble de générateurs de l'anneau gradué $R^K(G\wr S_\infty)$ :
	$$\mathbb{Z}[X_{k,i} | k\in [\![1;N]\!], i\geq 1]  = \mathbb{Z}[\Phi_k(x_i) | k\in [\![1;N]\!], i\geq 1]$$
\end{lemme}
\begin{proof}
	On le montre par récurrence sur $n\in \mathbb{N}^*$ que $\mathbb{Z}[X_{k,i} | k\in [\![1;N]\!], 1\leq i \leq n]  = \mathbb{Z}[\Phi_k(x_i) | k\in [\![1;N]\!], 1\leq i \leq n]$.
	Pour $i=1$, on a :
	$$X_{1,1}=\varphi_{1,1}\Phi_1(x_1) + \ldots + \varphi_{N,1}\Phi_N(x_1)$$
	$$\vdots $$
	$$X_{N,1}=\varphi_{1,N}\Phi_1(x_1) + \ldots + \varphi_{N,N}\Phi_N(x_1)$$
	Mais $(\varphi_{i,j})_{1\leq i,j \leq N}$ est une matrice inversible dans $\mathbb{Z}$, donc 
	$$\mathbb{Z}[X_{k,1} | k\in [\![1;N]\!]]  = \mathbb{Z}[\Phi_k(x_1) | k\in [\![1;N]\!]]$$
	Supposons maintenant que $$\mathbb{Z}[X_{k,i} | k\in [\![1;N]\!], 1\leq i \leq n-1]  = \mathbb{Z}[\Phi_k(x_i) | k\in [\![1;N]\!], 1\leq i \leq n-1]$$
	On a 
	$$X_{1,n}=\varphi_{1,1}\Phi_1(x_n) + \ldots + \varphi_{N,1}\Phi_N(x_n) + R_1$$
	$$\vdots $$
	$$X_{N,n}=\varphi_{1,N}\Phi_1(x_n) + \ldots + \varphi_{N,N}\Phi_N(x_n) + R_N$$
	avec $R_k \in \mathbb{Z}[X_{k,i}| k\in [\![1;N]\!], 1\leq i \leq n-1]$ (on a même $R_k=0$ pour $M+1\leq k \leq N$). Ainsi les
	$$\varphi_{1,1}\Phi_1(x_n) + \ldots + \varphi_{N,1}\Phi_N(x_n)$$
	$$\vdots $$
	$$\varphi_{1,N}\Phi_1(x_n) + \ldots + \varphi_{N,N}\Phi_N(x_n)$$
	sont dans $\mathbb{Z}[X_{k,i}| k\in [\![1;N]\!], 1\leq i \leq n]$. Comme $(\varphi_{i,j})_{1\leq i,j \leq N} \in GL_N(\mathbb{Z})$, on en déduit que $\{\Phi_k(x_n)\}_{1\leq k\leq N}$ sont dans $\mathbb{Z}[X_{k,i}| k\in [\![1;N]\!], 1\leq i \leq n]$ et le résultat est prouvé.
\end{proof}

On peut maintenant donner le théorème principal de cette section :

\begin{ThmPrinc}L'anneau $K_0^k(G\wr S_\infty)$ est un anneau de polynômes. De plus on a :
	$$ \varepsilon(K_0^k(G\wr S_\infty)) = \mathbb{Z}[y_{i,n} | i\in [\![1;M]\!], p\nmid n] $$
\end{ThmPrinc}
\begin{proof}
	D'après le corollaire \ref{cor},
	$$ \mathbb{Z}[y_{i,n} | i\in [\![1;M]\!], 2\nmid n] \subset \varepsilon(K_0^k(G\wr S_\infty))$$
	De plus, en comparant les dimensions degré par degré, on a :
	\begin{equation}\label{beta}
	K\otimes \varepsilon(K_0^k(G\wr S_\infty)) = K[y_{i,n} | i\in [\![1;M]\!], p\nmid n]
	\end{equation}
	Il reste à montrer que pour $z$ homogène de degré $n$ dans $\varepsilon(K_0^k(G\wr S_\infty))$, $z$ est un polynôme à coefficients dans $\mathbb{Z}$ en les $y_{i,k}$, $i\in[\![1,M]\!]$ et $k\leq n$.
	
	 Considérons donc un tel $z$. D'après (\ref{beta}) :
	$$ z = \sum_{\lambda_1,\ldots,\lambda_M \in\mathcal{P}_{\text{reg}}, \sum |\lambda_i| = n}\alpha_{\lambda_1,\ldots,\lambda_M} y_{1,\lambda_1}\ldots y_{M,\lambda_M}$$
	avec $\alpha_{\lambda_1,\ldots,\lambda_M} \in K$, et
	$$ z = \sum_{\lambda_1,\ldots,\lambda_N \in\mathcal{P}, \sum |\lambda_i| = n}\beta_{\lambda_1,\ldots,\lambda_N} X_{1,\lambda_1}\ldots X_{N,\lambda_N}$$
	avec $\beta_{\lambda_1,\ldots,\lambda_N} \in \mathbb{Z}$, d'après le lemme \ref{chtgen}.
	Or $y_{i,\lambda}=X_{i,\lambda} + R_{i,\lambda}$
	où $R_{i,\lambda}$ est combinaison linéaire de $X_{i,\lambda'}$ avec $\lambda' \notin \mathcal{P}_{\text{reg}}$ et la famille $\{X_{1,\lambda_1}\ldots X_{N,\lambda_N}\}_{\lambda_i\in\mathcal{P}, \sum |\lambda_i| =n }$ forme une base de $R^K(G\wr S_\infty)_n$.
	 Donc
	$$y_{1,\lambda_1}\ldots y_{M,\lambda_M}=X_{1,\lambda_1}\ldots X_{M,\lambda_M} + R_{\lambda_1,\ldots,\lambda_M}$$
	où $R_{\lambda_1,\ldots,\lambda_M}$ est un polynôme en les $X_{i,\lambda'}$ avec $i\in[\![1,M]\!]$ et  $\lambda'\notin\mathcal{P}_{\text{reg}}$.
 On a alors :
	\begin{eqnarray*}
		0 & = & \sum_{\substack{\lambda_1,\ldots,\lambda_M \in\mathcal{P}_{\text{reg}},\\ \sum |\lambda_i|=n}} (\alpha_{\lambda_1,\ldots,\lambda_M} - \beta_{\lambda_1,\ldots,\lambda_M})X_{1,\lambda_1}\ldots X_{M,\lambda_M} \\
		&  &+ \sum_{\substack{\lambda_1,\ldots,\lambda_M \in\mathcal{P}_{\text{reg}},\\ \sum |\lambda_i|=n}} \alpha_{\lambda_1,\ldots,\lambda_M}R_{\lambda_1,\ldots,\lambda_M} \\
		&  & -\sum_{\substack{\lambda_1,\ldots,\lambda_M\notin\mathcal{P}_{\text{reg}},\\ \sum |\lambda_i|=n}} \beta_{\lambda_1,\ldots,\lambda_M}X_{1,\lambda_1}\ldots X_{N,\lambda_M} - \sum_{\substack{(\lambda_{M+1},\ldots,\lambda_N)\neq 0,\\ \sum |\lambda_i|=n}} \beta_{\lambda_{1},\ldots,\lambda_N}X_{1,\lambda_1}\ldots X_{N,\lambda_N}
	\end{eqnarray*}
	Par indépendance linéaire des $\{X_{1,\lambda_1}\ldots X_{N,\lambda_N}\}_{\lambda_i\in\mathcal{P}, \sum |\lambda_i| =n }$, on en déduit que pour tout $\lambda_1,\ldots,\lambda_M\in\mathcal{P}_{\text{reg}}$, $\alpha_{\lambda_1,\ldots,\lambda_M}=\beta_{\lambda_1,\ldots,\lambda_M}$ donc $\alpha_{\lambda_1,\ldots,\lambda_M}\in\mathbb{Z}$ et le résultat est prouvé.
\end{proof}
\bibliographystyle{alpha}
\nocite{Zel}
\nocite{James}
\nocite{Geis}
\nocite{GeisKinch}
\nocite{JamesKerber}
\nocite{Serre}
\nocite{Webb}
\nocite{Cartier}
\nocite{CurtisReiner1}
\nocite{Vogel}
\nocite{GinbergReiner}
\bibliography{biblipolyrep}
\end{document}